\documentclass[a4paper,10pt]{svjour3}

\usepackage[english]{babel}
\usepackage[sc]{mathpazo}
\usepackage{amssymb,amsmath,paralist,amssymb,mathrsfs}
\usepackage{booktabs,microtype,xspace,array,fixltx2e,fancyhdr}
\usepackage{eucal,stmaryrd,enumerate,amscd}
\usepackage{graphicx,subfigure}
\usepackage{tikz}
\usepackage{esvect}
\usetikzlibrary{matrix,arrows}
\usepackage{pb-diagram}
\usepackage{amsfonts}
\usepackage{color}

\DeclareFontFamily{OT1}{pzc}{}
\DeclareFontShape{OT1}{pzc}{m}{it}{<-> s * [1.25] pzcmi7t}{}
\DeclareMathAlphabet{\mathpzc}{OT1}{pzc}{m}{it}

\newcommand{\Hodge}{{*}}
\newcommand{\SU}{\mathrm{SU}}
\newcommand{\GT}{\mathrm{G}_2}
\newcommand{\SO}{\mathrm{SO}}
\newcommand{\so}{\mathfrak{so}}
\newcommand{\su}{\mathfrak{su}}

\newcommand{\SL}{\mathrm{SL}}
\newcommand{\bR}{\mathbb R}
\newcommand{\bC}{\mathbb C}
\newcommand{\NB}{\nabla}
\newcommand{\LC}{\NB^{\textup{LC}}}
\newcommand{\scal}{\mathpzc{s}}

\newcommand{\Lmb}{\Lambda}
\newcommand{\lmb}{\lambda}
\newcommand{\OP}{\oplus}

\newcommand{\SM}{ S^3\times S^3}

\DeclareMathOperator{\tr}{tr}
\DeclareMathOperator{\diag}{diag}
\DeclareMathOperator{\Hol}{Hol}
\DeclareMathOperator{\Adj}{Adj}
\DeclareMathOperator{\Ric}{Ric}

\newcolumntype{C}{>{$}c<{$}}

\begin{document}
               
\parskip1pt

\thispagestyle{empty}

\begin{center}
  \LARGE\bfseries Half-flat structures on \boldmath{\(\SM\)}
\end{center}
\begin{center}
  \large Thomas Bruun Madsen and Simon Salamon
\end{center}

\begin{abstract}
  We describe left-invariant half-flat \( \SU(3) \)-structures on \(
  \SM \) using the representation theory of \( \SO(4) \) and matrix
  algebra. This leads to a systematic study of the associated
  cohomogeneity one Ricci-flat metrics with holonomy \( \GT \)
  obtained on \( 7 \)-manifolds with equidistant \( \SM \)
  hypersurfaces. The generic case is analysed numerically.
\end{abstract}

\bigskip
\begin{small}\noindent
\end{small}

\bigskip

\begin{small}\noindent
\end{small}
\bigskip

\section{Introduction}
\label{sec:intro}

It was Calabi \cite{Calabi:almcpl6} who first recognised the rich
geometry that can be found on a hypersurface of \( \bR^7 \) when the
latter is equipped with its natural cross product and \( \GT
\)-structure. The realization, much later, of metrics with holonomy
\emph{equal} to \( \GT \) allowed this theory to be extended, whilst
retaining the key features of the ``Euclidean'' theory. The second
fundamental form or Weingarten map \( W \) of a hypersurface \( Y \)
in a manifold \( X \) with holonomy \( \GT \) can be identified with
the intrinsic torsion of the associated \( \SU(3) \)-structure. The
latter is defined by a 2-form \( \omega \) and a \( 3 \)-form \(
\gamma \) induced on \( Y \), and \( W \) is determined by their
exterior derivatives. The symmetry of \( W \) translates into a
constraint on the intrinsic torsion (equivalently, on \( d\omega \)
and \( d\gamma \)) that renders the \( \SU(3) \)-structure what is
called \emph{half flat}.

Conversely, a \( 6 \)-manifold \( Y \) with an \( \SU(3) \)-structure
that is half flat can (at least if it is real analytic) be embedded in
a manifold with holonomy \( \GT \) \cite{Bryant:emb}. The metric \( g
\) on \( X \) is found by solving a system of evolution equations that
Hitchin \cite{Hitchin:stable} interpreted as Hamilton's equations
relative to a symplectic structure defined (roughly speaking) on the
space parametrising the pairs \( (\omega,\gamma) \). The simplest
instance of this construction occurs when \( Y \) is a so-called
\emph{nearly-K\"ahler} space, in which case \( g \) is a conical
metric over \( Y \), in accordance with a more general scheme
described by B\"ar \cite{Baer:spinor}. The first explicit metrics
known to have holonomy equal to \( \GT \) were realized in this way.

In this paper, we are concerned with the classification of
left-invariant half-flat \( \SU(3) \)-structures on \( \SM \),
regarded as a Lie group \( G \), up to an obvious notation of
equivalence. One of these structures is the nearly-K{\"a}hler one that
can be found on \( G\times G \), for any compact simple Lie group \( G
\), by realizing the product as the 3-symmetric space \( (G\times
G\times G)\slash G \). Indeed, we verify that this nearly-K\"ahler
structure is unique amongst invariant \( \SU(3) \)-structures on \(
\SM \) (see Proposition~\ref{prop:nKunique}, that has a dynamic
counterpart in Proposition~\ref{prop:NKunique}).

Examples of the resulting evolution equations for \( \GT \)-metrics
have been much studied in the literature
\cite{Brandhuber-al:G2,Cvetic-al:M3-G2,Cvetic-al:conifold}, but one of
our aims is to highlight those \( \GT \)-metrics that arise from
half-flat metrics with specific intrinsic torsion, motivated in part
by the approach in \cite{Butruille:W14}. Nearly-K\"ahler corresponds
to Gray-Hervella class \( \mathcal W_1 \), and it turns out that a
useful generalization in our half-flat context consists of those
metrics of class \( \mathcal W_1+\mathcal W_3 \); see
Section~\ref{sec:symgrp}. By careful choices of the coefficients in \(
\omega \) and \( \gamma \), we obtain metrics on \( \SM \) of the same
class with zero scalar curvature.

Another aim is to develop rigorously the algebraic structure of the
space of invariant half-flat structures on \( \SM \), and in
Section~\ref{sec:para} we show that the moduli space they define is
essentially a finite-dimensional symplectic quotient. This is a
description expected from \cite{Hitchin:stable}, and in our treatment
relies on elementary matrix theory. For example, the \( 2 \)-form \(
\omega \) can be represented by a \( 3\times3 \) matrix \( P \), and
mapping \( \omega \) to the 4-form \(
\delta=\omega^2=\omega\wedge\omega \) corresponds to mapping \( P \)
to the transpose of its adjugate. We shall
however choose to use a pair of symmetric \( 4\times4 \) matrices \( (Q,P) \) to
parametrise the pair \( (\gamma,\omega) \).

The matrix algebra is put to use in Section~\ref{sec:flow} to simplify
and interpret the flow equations for the associated Ricci-flat metrics
with holonomy \( \GT \). The significance of the class \( \mathcal
W_1+\mathcal W_3 \) becomes clearer in the evolutionary setting, as it
generates known \( \GT \)-metrics. In our formulation, the equations
(for example in Corollary~\ref{cor:flow}) have features in common with
two quite different systems considered in
\cite{FeraPontov-al:Painleve} and \cite{Dancer-W:painleve}, but both
in connection with Painlev\'e equations.

A more thorough analysis of classes of solutions giving rise to \( \GT
\)-metrics is carried out in Section~\ref{sec:further}. Some of these
exhibit the now familiar phenomenon of metrics that are asymptotically
circle bundles over a cone (``ABC metrics''). All our \( \GT
\)-metrics are of course of cohomogeneity one, and this allows us to
briefly relate our approach to that of \cite{Dancer-W:superpot}.

In the final part of the paper, we present the tip of the iceberg that
represents a numerical study of Hitchin's evolution equations for \(
\SM \). We recover metrics that behave asymptotically locally
conically when \( Q \) belongs to a fixed \( 2 \)-dimensional
subspace. More precisely, we show empirically that the planar
solutions are divided into two classes, only one of which is of type
ABC. This can be understood in terms of the normalization condition
that asserts that \( \omega \) and \( \gamma \) generate the same
volume form, and is a worthwhile topic for further theoretical
study. For the generic case, the flow solutions do not have tractable
asymptotic behaviour, but again the geometry of the solution curves
(illustrated in Figure~\ref{fig:G2sol3D}) is constrained by the
normalization condition that defines a cubic surface in space.

This paper grew out of an attempt to reconcile various contributions
appearing in the literature. Of particular importance concerning \(
\SU(3) \)-structures are Schulte-Hengesbach's classifications of half-flat
structures \cite[Theorem 1.4, Chapter 5]{Hengesbach:phd}, and Hitchin's notion of stable forms
\cite{Hitchin:stable}.  In addition, the explicit constructions of \(
\GT \)-metrics appearing in this paper are based on the work of
Brandhuber et al, Cveti{\v{c}} et al
\cite{Brandhuber-al:G2,Cvetic-al:M3-G2,Cvetic-al:conifold}, as well as
the contributions of Dancer and Wang \cite{Dancer-W:painleve}.

\section{Invariant \boldmath{\(\SU(3)\)}-structures}
\label{sec:symgrp}

Throughout the paper \( M \) will denote the \( 6 \)-manifold \(
S^3\times S^3 \). As this is a Lie group, we can trivialise the
tangent bundle.  We describe left-invariant tensors via the
identification
\[ TM\cong M\times \so(4)\cong M\times\bR^6, \] relative to left
multiplication. We keep in mind that there are Lie algebra
isomorphisms
\[ \su(2)\OP\su(2)\cong\so(3)\OP\so(3)\cong\so(4), \] which at the
group level can be phrased in terms of the diagram
\begin{equation}
  \label{eq:grpseq}
  \begin{diagram}
    \node{\SU(2)^2} \arrow{e,l}{2:1} \arrow{se,b} {4:1} \node {\SO(4)} \arrow{s,r} {2:1} \\
    \node{} \node { \SO(3)^2}
  \end{diagram}
\end{equation}

The cotangent space of \( M \), at the identity, consists of two
copies of \( \su(2)^* \). We shall write \( T^*=T_1^*M=A\OP B \) and
choose bases \( e^1,e^3,e^5 \) of \( A \) and \( e^2,e^4,e^6 \) of \(
B \) such that
\begin{equation}
  \label{eq:stdbasis}
  de^1=e^{35},\, de^2=e^{46}, \textrm{ and so forth}; 
\end{equation}
here \( d \) denotes the exterior differential on \( A \) and \( B \)
induced by the Lie bracket.

We wish to endow \( M \) with an \( \SU(3) \)-structure. To this end
it suffices to specify a suitable pair of real forms: a \( 3 \)-form
\( \gamma \), whose stabiliser (up to a \( \mathbb Z\slash2
\)-covering) is isomorphic to \( \mathrm{SL}(3,\bC) \), and a
non-degenerate real \( 4 \)-form \( \delta=\omega\wedge\omega=\omega^2
\). These two forms must be compatible in certain ways. Above all, \(
\gamma \) must be a \emph{primitive} form relative to \( \omega \),
meaning \( \gamma \wedge\omega=0 \). So as to obtain a genuine almost
Hermitian structure we also ask for volume matching and positive
definiteness:
\begin{equation}
  \label{eq:comp-halfflat}
  3\gamma\wedge\hat\gamma=2\omega^3,\quad  \omega(\cdot,J\cdot)>0. 
\end{equation}

These forms \( \gamma \) and \( \delta \) are \emph{stable} in the
sense their orbits under \( \mathrm{GL}(6,\bR) \) are open in \(
\Lmb^kT^* \). The following well known properties
(cf. \cite{Hitchin:stable}, and \cite{Reichel:3forms,Westwick:3forms} for
the study of \( 3 \)-forms) of stable forms will be used in the
sequel:
\begin{enumerate}
\item There are two types of stable \( 3 \)-forms on \( T \). These
  are distinguished by the sign of a suitable quartic invariant, \(
  \lmb \), which is negative precisely when the stabiliser is \(
  \SL(3,\bC) \) (up to \( \mathbb Z\slash2 \)); each form of this
  latter type determines an almost complex structure \( J \).
\item The stable forms \( \delta \) and \( \gamma \) determine
  ``dual'' stable forms: \( \delta \) determines the stable \( 2
  \)-form \( \pm\omega \), and \( \gamma \) determines the \( 3
  \)-form \( \hat\gamma=J(\gamma) \) characterised by the condition
  that \( \gamma+i\hat\gamma \) be of type \( (3,0) \).
\end{enumerate}

As \( \SU(3) \)-modules \( \Lmb^kT^* \) decomposes in the following
manner:
\begin{equation}
  \begin{gathered}
    T^*\cong[\![\Lmb^{1,0}]\!]\cong\Lmb^5T^*,\\
    \Lmb^2T^*\cong[\![\Lmb^{2,0}]\!]\OP[\Lmb^{1,1}_0]\OP\bR\cong\Lmb^4T^*,\\
    \Lmb^3T^*\cong[\![\Lmb^{3,0}]\!]\OP[\![\Lmb^{2,1}_0]\!]\OP[\![\Lmb^{1,0}]\!],
  \end{gathered}
\end{equation}
using the bracket notation of \cite{Sal:Redbook}. In terms of this
decomposition (see \cite{Bedulli-V:SU3}), the exterior derivatives of
\( \gamma,\omega \) may now be expressed as
\begin{equation*}
  \begin{cases}
    d\omega=-\frac32w_1\gamma+\frac32\hat{w}_1\hat\gamma+w_4\wedge \omega+w_3,\\
    d\gamma=\hat w_1\omega^2+w_5\wedge\gamma+w_2\wedge\omega,\\
    d\hat\gamma=w_1\omega^2+(Jw_5)\wedge\gamma+\hat{w}_2\wedge\omega,
  \end{cases}
\end{equation*}
where we have used a suggestive notation to indicate the relation
between forms and the intrinsic torsion \( \tau \), i.e., the failure
of \( \Hol(\LC) \) to reduce to \( \SU(3) \). Obviously, this
expression depends on our specific choice of normalisation
(cf.~\eqref{eq:comp-halfflat}).

Generally, \( \tau \) takes values in the \( 42 \)-dimensional space
\[ T^*\otimes\su(3)^\perp\cong\mathcal W_1\OP\mathcal W_2\OP\mathcal
W_3\OP\mathcal W_4\OP\mathcal W_5. \] Our main focus, however, is to
study the subclass of \emph{half-flat \( \SU(3) \)-structures}: these
are characterised by the vanishing of \( \hat w_1, w_2,w_4 \), and \(
w_5 \), i.e.,
\begin{equation*}
  \begin{cases}
    d\omega=-\frac32w_1\gamma+w_3,\\
    d\gamma=0,\\
    d\hat\gamma=w_1\omega^2+\hat w_2\wedge\omega.
  \end{cases}
\end{equation*}

\begin{remark}
  To appreciate the terminology ``half flat'', it helps to count
  dimensions: \( \dim\mathcal W_1=2 \), \( \dim\mathcal W_2=16 \), \(
  \dim \mathcal W_3=12 \), \( \dim\mathcal W_4=6=\dim\mathcal W_5
  \). In particular, observe that for half-flat structures \( \tau \)
  is restricted to take its values in \( 21 \) dimensions out of \( 42
  \) possible. In this context, ``flat'' would mean \emph{\(
    \mathrm{SU}(3) \) holonomy}.

\end{remark}

For emphasis, we formulate:

\begin{proposition}
  \label{prop:intr_space}
  For any invariant half-flat \( \SU(3) \)-structure \(
  (\omega,\gamma) \) on \( M \) the following holds:
  \begin{compactenum}
  \item if \( \mathcal W_3=0 \) then \( d\omega=-\frac32w_1\gamma \).
  \item if \( \mathcal W_2^-=0 \) then \( d\hat\gamma=w_1\omega^2 \).
  \end{compactenum}
  In particular, any structure with vanishing \( \mathcal W_3 \)
  component has \( [\gamma]=0\in H^3(M) \).  \qed
\end{proposition}

In the case when \( \mathcal W_3=0 \) we shall say the half-flat
structure is \emph{coupled}. The second case above, \( \mathcal
W_2^-=0 \), is referred to as \emph{co-coupled}. When the half-flat
structure is both coupled and co-coupled, so \( \mathcal
W_2^-=0=\mathcal W_3 \), it is said to be \emph{nearly-K\"ahler}.

\paragraph{Examples of type \boldmath{\( \mathcal W_1+\mathcal W_3
    \)}.}

As the next two examples illustrate, it is not difficult to construct
half-flat structures of type \( \mathcal W_1+\mathcal W_3 \).

\begin{example}
  \label{ex:W1W3}
  In this example we fix a non-zero real number \( a\in\bR^* \) and
  consider the pair of forms \( (\omega,\gamma) \) given by:
  \begin{equation*}
    \begin{cases}
      \omega=-\frac34\alpha a\left(e^{12}+e^{34}+e^{56}\right),\\
      \gamma=a(e^{135}-e^{246})+\frac12a\left(e^{352}-e^{146}+e^{514}-e^{362}+e^{136}-e^{524}\right),
    \end{cases}
  \end{equation*}
  where \( \alpha \) is defined via the relation
  \begin{equation*}
    \frac{a\alpha^3}{2\sqrt{3}}=\frac49.
  \end{equation*}
  Clearly, \( d(\omega^2)=0 \) and \( d\gamma=0 \).

  A calculation shows \( \lmb=-\frac{27}{16}a^4 \) so that
  \[ \sqrt{-\lmb}=\frac{3\sqrt{3}}4a^2. \] The \( 3 \)-form \(
  \hat{\gamma} \) is given by
  \begin{equation*}
    \hat{\gamma}=-\frac{\sqrt{3}}2a\left(e^{352}+e^{146}+e^{514}+e^{362}+e^{136}+e^{524}\right).
  \end{equation*}

  Note that the following normalisation condition is satisfied:
  \begin{equation*}
    \frac23\omega^3=-\frac{27\alpha^3a^3}{16}e^{123456}=-\frac{9\alpha^3}{4}\frac{3a^3}4e^{123456}=-\frac{3\sqrt{3}a^2}2e^{123456}=\gamma\wedge\hat\gamma.
  \end{equation*}

  In order to verify that the intrinsic torsion is of type \( \mathcal
  W_1+\mathcal W_3 \), we calculate the exterior derivatives of \(
  \omega \), \(\gamma \), and \( \hat\gamma \):
  \begin{equation*}
    \begin{cases}
      d\omega=-\frac32\alpha\gamma+\frac32\alpha a(e^{135}-e^{246}),\\
      d\gamma=0,\\
      d\hat\gamma=\alpha\omega^2.
    \end{cases}
  \end{equation*}

  Finally, note that the associated metric is given by
  \[ g=\frac{\sqrt{3}}2\alpha a\sum_{i=1}^3\left(e^{2i-1}\otimes
    e^{2i-1}+e^{2i}\otimes e^{2i}+\frac12(e^{2i-1}\otimes e^{2i}+e^{2i}\otimes
    e^{2i-1})\right),\] and one finds that the scalar curvature is
  positive: \( \mathpzc{s}=\frac4{\sqrt{3}\alpha a}=\frac32\alpha^2
  \).
\end{example}

\begin{example}[Zero scalar curvature metric]
  \label{ex:W1W3s0}
  Consider the following pair of stable forms:
  \begin{equation*}
    \begin{cases}
      \omega=a\left(e^{12}+e^{34}+e^{56}\right),\\
      \gamma=\sqrt{5}b(e^{135}-e^{246})+b\left(e^{352}-e^{146}+e^{514}-e^{362}+e^{136}-e^{524}\right),
    \end{cases}
  \end{equation*}
  
  We find that \( \lmb=-8(1+\sqrt{5})b^4 \), and the \( 3 \)-form \(
  \hat{\gamma} \) is given by
  \begin{equation*}
    \begin{split}
      - {\sqrt{-\lmb}} \hat{\gamma}&=2(\sqrt{5}-1)b^3(e^{135}+e^{246})\\
      &\qquad+2(3+\sqrt{5})b^3\left(e^{352}+e^{146}+e^{514}+e^{362}+e^{136}+e^{524}\right).
    \end{split}
  \end{equation*}
  The normalisation condition then reads
  \[ a^3=-\sqrt{2(1+\sqrt{5})}b^2. \]

  The associated metric takes the form

  \[
  g=-\frac{2ab^2}{\sqrt{-\lmb}}\sum_{i=1}^3\left((1+\sqrt{5})(e^{2i-1}\otimes
    e^{2i-1}+e^{2i}\otimes e^{2i})+2(e^{2i-1}\otimes e^{2i}+e^{2i}\otimes
    e^{2i-1})\right).\] In this case one finds that the scalar curvature is
  zero.
\end{example}

\begin{remark}[Group contractions]
  The author of \cite{Conti:SU3} uses Lie algebra degenerations to
  study invariant hypo \( \SU(2) \)-structures on \( 5 \)-dimensional
  nilmanifolds. In a similar way, one could study half-flat structures
  on the various group contractions of \( \SM \) like \( S^3\times N^3
  \), where \( N^3 \) is a compact quotient of the Heisenberg
  group. (See \cite{Chong-al:G2contr} for partial studies of such
  contractions).
\end{remark}

\section{Parametrising invariant half-flat structures}
\label{sec:para}

The invariant half-flat structures on \( M \) can be described in
terms of symmetric matrices. In order to do this, we recall the local
identifications \eqref{eq:grpseq} and set \( U=\bR^{3,3} \), the space
of real \( 3\times3 \) matrices, and \( V=S^2_0(\bR^4) \), the space
of real symmetric trace-free \( 4\times4 \) matrices.

There is a well known correspondence between \( U \) and \( V \); a
fact which is for example used in the description of the trace-free
Ricci-tensor \( \Ric_0\in\Lmb^2_+\otimes\Lmb^2_- \) on a Riemannian \(
4 \)-manifold.

\begin{lemma}
  \label{lem:equiv-iso}
  There is an equivariant isomorphism \( U\to V \) which maps a \( 3
  \times 3 \) matrix \( K=(k_{ij}) \) to the matrix
  \begin{gather*}
    \left(\begin{array}{cccc}
        -k_{11}-k_{22}-k_{33}&k_{23}-k_{32}&-k_{13}+k_{31}&k_{12}-k_{21}\\
        k_{23}-k_{32}&-k_{11}+k_{22}+k_{33}&-k_{12}-k_{21}&-k_{13}-k_{31}\\
        -k_{13}+k_{31}&-k_{12}-k_{21}&k_{11}-k_{22}+k_{33}&-k_{23}-k_{32}\\
        k_{12}-k_{21}&-k_{13}-k_{31}&-k_{23}-k_{32}&k_{11}+k_{22}-k_{33}
      \end{array}\right).
  \end{gather*}
\end{lemma}
\begin{proof}
  By fixing an oriented orthonormal basis \( \{f_1,f_2,f_3,f_4\} \) of
  \( (\bR^4)^* \), we make the identifications \( \Lmb^2_+=A \), \(
  \Lmb^2_-=B \) via
  \[ e^1=f^{12}+f^{34},\,e^2=f^{12}-f^{34}, \textrm{ and so forth.} \]
  The asserted isomorphism is then given by contraction on the middle
  two indices, as in the following example:
  \begin{equation*}
    \begin{split}
      U\cong A\otimes B\ni e^5\otimes e^2&=(f^{14}+f^{23})\otimes(f^{12}-f^{34})\\
      &=(f^1f^4-f^4f^1+f^2f^3-f^3f^2)(f^1f^2-f^2f^1-f^3f^4+f^4f^3)\\
      &\longmapsto f^1f^3-f^4f^2-f^2f^4+f^3f^1=f^1\odot f^3-f^2\odot
      f^4\in V.
    \end{split}
  \end{equation*}
\end{proof}

Table \ref{tab:comp-inv-cov} summarises how invariants and covariants
are related under the above isomorphism \( U\cong V \).

\bigbreak

\begin{table}[htp]
  \centering
  \begin{tabular}{CC}
    \toprule
    \vspace{0.1mm}
    K\in U & S\in V
    \vspace{0.2mm}\\
    \hline
    \vspace{0.2mm}
    K & S
    \vspace{0.2mm}\\
    \hline
    \vspace{0.2mm}
    4\tr(KK^T) & \tr(S^2)
    \vspace{0.1mm}\\
    -2\Adj(K^T) & (S^2)_0
    \vspace{0.2mm}\\
    \hline
    \vspace{0.2mm}
    -24\det(K) & \tr(S^3)
    \vspace{0.1mm}\\
    4\tr(KK^T)K & \tr(S^2)S
    \vspace{0.1mm}\\
    2KK^TK & \frac34\tr(S^2)S-(S^3)_0
    \vspace{0.2mm}\\
    \hline 
    \vspace{0.2mm}
    4\tr((KK^T)^2) & 3\det(S)+\frac14 \tr(S^4)
    \vspace{0.1mm}\\
    2\tr(KK^T)^2 &\det(S)+\frac14\tr(S^4)
    \vspace{0.1mm}\\
    -24\det(K)K & \tr(S^3)S
    \vspace{0.1mm}\\
    4\tr(KK^T)\Adj(K) & \frac13\tr(S^3)S-(S^4)_0\\ 
    \bottomrule
  \end{tabular}
  \caption{Dictionary between invariants and covariants; \( S \) denotes the image of \( K \) under the isomorphism \( U \to V \) of Lemma \ref{lem:equiv-iso}.}
  \label{tab:comp-inv-cov}
\end{table}  

Now, let us fix a cohomology class \( c=(a,b)\in H^2(M,\bR)\cong\bR^2
\).  We have:

\begin{theorem}
  \label{thm:half-flat-param}
  The set \( \mathcal H_c \) of invariant half-flat structures on \( M
  \) with \( [\gamma]=c \) can be regarded as a subset of the
  \emph{commuting variety}:
  \begin{equation}
    \label{eq:comm-var}
    \left\{(Q,P)\in V\OP V\colon\,[Q,P]=0\right\}.
  \end{equation}
\end{theorem}
\begin{proof}
  Recall \( T^*M=A\OP B \), where \( A\cong \su(2)^*\cong B \) so that
  we have
  \begin{gather*}
    \Lmb^2T^*\cong \Lmb^2A\OP(A\otimes B)\OP \Lmb^2V\cong\Lmb^4T^*M\\
    \Lmb^3T^*\cong \Lmb^3A\OP(\Lmb^2A\otimes
    B)\OP(A\otimes\Lmb^2B)\OP\Lmb^3B.
  \end{gather*}

  The equation \( d(\omega^2)=0 \) implies that
  \[ \omega\in A\otimes B\cong U\cong V, \] which defines \( P
  \). Also note \( \delta=\omega^2 \) lies in a space isomorphic to \(
  V \).

  We may assume that
  \[\gamma=ae^{135}+d\beta+be^{246}\]

  The condition \( \omega\wedge \gamma=0 \) implies \( Q\otimes P \)
  lies in the kernel of some \( \SO(4) \)-equivariant map
  \[ V\otimes V\longrightarrow \Lmb^5T^*M\cong A\OP B\cong
  \Lmb^2\bR^4, \] which must correspond to \( [Q,P]=QP-PQ \).
\end{proof}

\begin{remark}
  Consider the open subset set \( \mathcal U_c \), \(c=(a,b) \), of
  the commuting variety given by pairs \( (Q,P) \) satisfying
  \begin{equation}
    \tr(P^3)\neq0,\quad \det(Q)+\frac{a-b}6\tr(Q^3)+\frac{ab}2\tr(Q^2)+(ab)^2<0.
  \end{equation}
  Then \( \mathcal H_c \) is the hypersurface in \( \mathcal U_c \)
  characterised by the normalisation condition
  \begin{equation}
    \label{eq:normMatr}
    \tr(P^3)=12\left(-\det(Q)-\frac{a-b}6\tr(Q^3)-\frac{ab}2\tr(Q^2)-(ab)^2\right)^{\frac12}.
  \end{equation}
\end{remark}

The space \( V\OP V\cong V\times V^* =T^*V \) has a natural symplectic
structure, and \( \SO(4) \) acts Hamiltonian with moment map \(
\mu\colon\, V\OP V\rightarrow \mathfrak{so}(4)\cong\Lmb^2\bR^4\) given
by
\[ (Q,P)\longmapsto[Q,P]. \] Via (singular) symplectic reduction
\cite{Lerman:reduc}, we can the simplify the parameter space
significantly:
\begin{corollary}
  \label{cor:param}
  The set \( \mathcal H_c \) of half-flat structures modulo
  equivalence relations is a subset of the singular symplectic
  quotient
  \[
  \frac{\mu^{-1}(0)}{\SO(4)}\cong\frac{\bR^3\OP\bR^3}{S_3}. \] \qed
\end{corollary}

For later use, we observe that in terms of the matrix framework, the
dual \( 3 \)-form \( \hat\gamma \) has exterior derivative given as
follows:

\begin{lemma}
  \label{lem:hatgamma}
  Fix a cohomology class \( c=(a,b)\in H^3(M) \). For any element \(
  (Q,P)\in\mathcal H_c \) corresponding to an invariant half-flat
  structure, the associated \( 4 \)-form \( d\hat\gamma \) corresponds
  to the matrix \( \hat R=\frac1{\sqrt{-r}} R \), where
  \begin{equation*}
    \begin{cases}
      R= -(Q^3)_0+\frac{a-b}2(Q^2)_0+(ab+\frac12\tr(Q^2))Q,\\
      4r=\det(Q)+\frac{a-b}6\tr(Q^3)+\frac{ab}2\tr(Q^2)+(ab)^2\,(=\lambda(c,Q))
    \end{cases}
  \end{equation*} 

  In particular, if \( a+b=0 \) and we set \( \hat Q=Q+a I \) then
  \begin{equation*}
    \begin{cases}
      R= (\Adj(\hat Q))_0,\\
      4r=\det(\hat Q)
    \end{cases}
  \end{equation*}
\end{lemma}

\begin{proposition}
  Let \( (Q,P)\in \mathcal H_c \):
  \begin{compactenum}
  \item if \( (Q,P) \) corresponds to a coupled structure then \( c=0
    \) and \( P=-\frac32\alpha Q \) for a non-zero constant \(
    \alpha\in\bR \).
  \item if \( (Q,P) \) corresponds to a co-coupled structure then \(
    \hat R=\alpha (P^2)_0 \) for a non-zero constant \( \alpha\in\bR
    \).
  \end{compactenum}
  \qed
\end{proposition}

\begin{example}
  Obviously, the half-flat pair \( (Q,P) \) is of type \( \mathcal
  W_1+\mathcal W_3 \) if and only if the matrices \( (P^2)_0 \) and \(
  R \) are proportional, i.e., we have \( \hat R=\alpha (P^2)_0 \);
  the type does not reduce further provided \( c\neq0 \) and \(
  \alpha\neq0 \). Using these conditions it is easy to show that the
  structures of Example \ref{ex:W1W3} and Example \ref{ex:W1W3s0} have
  the type of intrinsic torsion claimed. Indeed, in the first example,
  using Lemma \ref{lem:hatgamma}, we find that
  \begin{equation*}
    (P^2)_0=\frac{9a^2\alpha^2}{8}\diag(3,-1,-1,-1),\quad R=\frac{9a^3}{8}\diag(3,-1,-1,-1),
  \end{equation*}
  whilst the matrices of the second example satisfy
  \begin{equation*}
    (P^2)_0=2a^2\diag(3,-1,-1,-1),\quad R=(\frac12\sqrt{5}a^2 b + 6b^3)\diag(3,-1,-1,-1).
  \end{equation*}
\end{example}

\begin{example}[Nearly-K\"ahler]
  \label{ex:nK}
  In this case, the following conditions should be satisfied:
  \begin{equation*}
    \begin{cases}
      P=-\frac32\alpha Q\equiv-\frac32\alpha\diag(-x-y-z,x,y,z),\\
      4\Adj(Q)_0=\sqrt{-\det(Q)}\alpha(P^2)_0=\frac94\alpha^3\sqrt{-\det(Q)}(Q^2)_0,
    \end{cases}
  \end{equation*}
  for some \( \alpha\in\bR^* \). This is equivalent to solving the
  equations
  \[ (Q^2)_0=\tilde \alpha\left((Q^3)_0-\frac12\tr(Q^2)Q\right), \]
  where \( \tilde\alpha=-\frac{16}{9\alpha^3\sqrt{-\det Q}} \). We
  find this system of equations can be formulated as
  \begin{equation*}
    \begin{cases}
      (y+z)(2x+y+z)=- \tilde\alpha yz(2x+y+z),\\
      (x+z)(x+2y+z)=-\tilde\alpha xz(x+2y+z),\\
      (x+y)(x+y+2z)=- \tilde\alpha xy(x+y+2z).
    \end{cases}
  \end{equation*}
  Keeping in mind that we must have \( (x+y+z)xy>0 \), we obtain only
  the following solutions \( (Q,P)\in \mathcal H_0 \):
  \begin{equation*}
    \begin{split}
      x&=y=z=\frac8{9\sqrt{3}\alpha^3}, \\
      -\frac13x&=y=z=\frac8{9\sqrt{3}\alpha^3} \quad\textrm{ or with
        the roles of } x,y,z \textrm{ interchanged}.
    \end{split}
  \end{equation*}
  Note that these solutions are identical after using a permutation;
  the corresponding matrices \( Q \) are of the form
  \[ \diag(-3x,x,x,x) \quad \textrm{and} \quad \diag(x,-3x,x,x),\]
  respectively.
\end{example}

The above example captures a well known fact about uniqueness of the
invariant nearly-K\"ahler structure on \( \SM \).  In our framework,
this can be summarised as follows (compare with \cite[Proposition
2.5]{Butruille:nK} and \cite[Proposition 1.11, Chapter 5]{Hengesbach:phd}).

\begin{proposition}
  \label{prop:nKunique}
  Modulo equivalence and up to a choice of scaling \( q\slash
  p\in\bR^* \), there is a unique invariant nearly-K\"ahler structure
  on \( M \). It is given by the class \( [(Q,P)] \) where
  \[ (Q,P)=(q(\diag(-3,1,1,1),p\diag(-3,1,1,1))\in \mathcal H_0. \]
  \qed
\end{proposition}

As observed in \cite[Proposition~1.8]{Hengesbach:phd} there are no
invariant (integrable) complex structures on \( M \) admitting a
left-invariant holomorphic \( (3,0) \)-form. Indeed, in terms of \(
4\times4\) matrices this assertion is captured by

\begin{lemma}
  In the notation of Lemma~\ref{lem:hatgamma}, if \( R=0 \) then \(
  r\geqslant0 \).  \qed
\end{lemma}

Although we have chosen to focus on the vector space \( V \) and \(
4\times4 \) matrices, we conclude this section with a neat consequence
of stability.  Consider \( K\in\bR^{3,3} \). The Cayley-Hamilton
theorem states that
\[ K^3-c_1K^2+c_2K-c_3I=0, \] where \( c_1=\tr K \), \(
\tr(K^2)=c_1^2-2c_2 \), and \( c_3=\det K \). Consider now the
adjugate
\[ \Adj K =K^2-c_1K+c_2I, \] so that \( K(\Adj K)=(\det K)I \).
Table~\ref{tab:comp-inv-cov} implies that the mapping \(
\omega\mapsto\omega^2 \) corresponds to a multiple of \( K\mapsto
\Adj(K^T) \). The following result describes a viable alternative to
the square root of a \( 3\times 3 \) matrix; it can be proved directly
using the singular value decomposition.

\begin{corollary} Any \( 3\times3 \) matrix with positive determinant
  equals \( \Adj K \) for some unique \( \pm K \).  \qed
\end{corollary}

\section{Evolution equations: from \boldmath{\(\SU(3)\)} to
  \boldmath{\(\GT\)}}
\label{sec:flow}

Let \( I\subset\bR \) be an interval. A \( \GT \)-structure and metric
on the \( 7 \)-manifold \( M\times I \) can be constructed from a
one-parameter family of half-flat structures on \( M \) by setting
\begin{equation}
  \label{eq:G2str}
  \begin{cases}
    \varphi=\omega(t)\wedge dt+\gamma(t),\\
    \Hodge\varphi=\hat\gamma(t)\wedge dt+\frac12\delta(t),
  \end{cases}
\end{equation}
where \( \delta(t)=\omega(t)^2 \) and \( t\in I \). It is well known
\cite{Fernandez-G:G2} the holonomy lies in \( \GT \) if and only if \(
d\varphi=0=d\Hodge\varphi \). For structures defined via a
one-parameter family of half-flat structures, this can be phrased
equivalently as:
 
\begin{proposition}
  The Riemannian metric associated with the \( \GT \)-structure
  \eqref{eq:G2str} has holonomy in \( \GT \) if and only if the family
  of forms satisfies the equations:
  \begin{equation}
    \label{eq:G2flow}
    \begin{cases}
      \gamma'=d\omega,\\
      \delta'=-2d\hat\gamma.
    \end{cases}
  \end{equation}
\end{proposition}
\begin{proof}
  Differentiation of \( \varphi \) and \( \Hodge\varphi \) gives us:
  \begin{equation*}
    \begin{cases}
      d\varphi=d\omega\wedge dt+d\gamma-\gamma'\wedge dt,\\
      d\Hodge\varphi=d\hat\gamma\wedge dt+\frac12d\delta+\delta'\wedge
      dt,
    \end{cases}
  \end{equation*}
  Since the one-parameter family consists of half-flat \( \SU(3)
  \)-structures, we have \( d\gamma=0=d\delta \) (for each fixed \( t
  \)), so the conditions \( d\varphi=0=d\Hodge\varphi \) reduce to the
  system \eqref{eq:G2flow}.
\end{proof}

\begin{remark}
  As explained in \cite[Theorem 8]{Hitchin:stable}, the evolution
  equations \eqref{eq:G2flow} can be viewed as the flow of a
  Hamiltonian vector field on \(
  \Omega^3_{ex}(M)\times\Omega^4_{ex}(M) \).  It is a remarkable fact
  that this flow does not only preserve the closure of \( \delta \)
  and \( \gamma \), but also the compatibility conditions
  \eqref{eq:comp-halfflat}.
\end{remark}

\begin{remark}
  In order to show that a given \( \GT \)-metric on \( M\times I \)
  has holonomy equal to \( \GT \), one must show there are no non-zero
  parallel \( 1 \)-forms on the \( 7 \)-manifold (see the treatment by
  Bryant and the second author \cite[Theorem
  2]{Bryant-S:exceptional}). For many of the metrics constructed in
  this paper, the argument is the same, or a variation of, the one
  applied in \cite[Section 3]{Bryant-S:exceptional}.
\end{remark}

In terms of matrices \( (Q,P)\in \mathcal H_c \), we can rephrase the
flow equations by

\begin{proposition}
  \label{prop:G2flow-matr}
  As a flow, \( t\mapsto (Q(t),P(t)) \), in \( \mathcal H_c \), the
  evolution equations \eqref{eq:G2flow} take the form
  \begin{equation}
    \label{eq:G2flow-matr}
    \begin{cases}
      Q'=P,\\
      (P^2)'_0=-2\hat R.
    \end{cases}
  \end{equation}
  \qed
\end{proposition}

These equations are particularly simple when the cohomology class \(
c=(a,b) \) of \( \gamma \) satisfies the criterion \( a+b=0 \). In
this case, by Lemma \ref{lem:hatgamma}, we have:

\begin{corollary}\label{cor:flow}
  For a flow, \( t\mapsto (Q(t),P(t)) \), in \( \mathcal H_{(a,b)} \)
  with \( a+b=0 \), the equations \eqref{eq:G2flow-matr} take the
  form:
  \begin{equation*}
    \begin{cases}
      Q'=P,\\
      (P^2)'_0=-\frac{4\Adj(\hat Q)_0}{\sqrt{-\det \hat Q}}.
    \end{cases}
  \end{equation*}
  \qed
\end{corollary}

\begin{remark}
  When phrased as above, the preservation of the normalisation
  \eqref{eq:normMatr} essentially amounts to Jacobi's formula for the
  derivative of a determinant.
\end{remark}

Proposition \ref{prop:G2flow-matr} tells us that the \( \GT \)-metrics
on \( M\times I \) that arise from the flow of invariant half-flat
structures, can be interpreted as the lift of suitable paths \(
t\mapsto Q(t) \) to paths
\[ t\mapsto(Q(t),P(t))\in S^2_0(\bR^4)\times S^2_0(\bR^4)\cong
T^*(S^2_0(\bR^4)),\] and moreover these paths lie on level sets of the
(essentially Hamiltonian) functional
\[ H_c(Q,P)=\sqrt{-\lambda(c,Q)}-\frac1{12}\tr(P^3). \]

\begin{corollary}
  Let \( (Q,P) \) be a (normalised) solution of the flow equations
  \eqref{eq:G2flow-matr}. Then the trajectory \( (Q(t),P(t)) \) lies
  on the level set \( \{H_c=0\} \) inside the space \(
  (S^2_0(\bR^4))^2\cong T^*(S^2_0(\bR^4)) \).  \qed
\end{corollary}

\paragraph{Dynamic examples of type \boldmath{\(\mathcal W_1+\mathcal W_3\)}.}
Rephrasing results of \cite{Brandhuber-al:G2}, we now consider the
one-parameter family of forms \( t\mapsto (\omega(t),\gamma(t)) \)
given by
\begin{equation*}
  \begin{cases}
    \omega(t)=-\frac32\alpha(t)x(t)(e^{12}+e^{34}+e^{56})\equiv-\frac32\alpha(t)x(t)\omega_0,\\
    \gamma(t)=x(t)d\omega_0+a(e^{135}-e^{246}).
  \end{cases}
\end{equation*}
 
In this case, we find that
\[ \lmb=(a-3x)(x+a)^3, \] and we shall assume \( 3x<a \) and \( x<-a
\), so as to ensure \( \lmb<0 \). Also note that
\begin{equation*}
  \begin{split}
    -\sqrt{-\lmb}\hat{\gamma}&=x(a+x)^2(e^{135}+e^{246})\\
    &\quad+(a-2x)(a+x)^2\left(e^{352}+e^{146}+e^{514}+e^{362}+e^{136}+e^{524}\right).
  \end{split}
\end{equation*}
In particular, the normalisation condition reads:
\begin{equation}
  \label{eq:norm-Hitchsol1}
  27\alpha^3x^3=4\sqrt{(3x-a)(x+a)^3}.
\end{equation} 
In order to solve the flow equations, we also need the \( 4 \)-form
\[ d\hat{\gamma}=\frac1{\sqrt{-\lmb}}x(x+a)^2\omega_0^2. \]

Based on the above expressions, the system \eqref{eq:G2flow} becomes:
\begin{equation*}
  \begin{cases}
    x'(t)=-\frac32\alpha(t)x(t),\\
    (\alpha^2x^2)'=-\frac{8}9x\sqrt{\frac{x+a}{3x-a}}.
  \end{cases}
\end{equation*}
These equations can be rewritten as a system of first order ODEs in \(
x \) and \( \alpha \):
\begin{equation*}
  \begin{cases}
    x'=-\frac32\alpha x\\
    \alpha'=\frac32\alpha^2-\frac49\frac1{\alpha
      x}\sqrt{\frac{x+a}{3x-a}}.
  \end{cases}
\end{equation*}
As we require the normalisation \eqref{eq:norm-Hitchsol1} to hold, we
cannot choose initial conditions \( (x_i,\alpha_i) \) freely.

After suitable reparametrization, we find the explicit solution:
\begin{equation}
  \begin{cases}
    x(s)=\frac13(4s^3+a),\\
    \alpha(s)=\frac{4s^2}{\sqrt{3}}\frac{\sqrt{1+as^{-3}}}{4s^3+a},
  \end{cases}
\end{equation}
where \( -\infty<s<\min\{0,-a^{\frac13}\} \), and
\[ t=-2\sqrt{3}\int\!\frac{ds}{\sqrt{1+as^{-3}}}.\]
  
Note that whilst \( x' \) is always non-zero, \( \alpha' \) can be
zero. Indeed, this happens if \( a \) is chosen such that the
quadratic equation
\[ x^2+2ax-a^2=0 \] has a solution \( x(s) \) for some \(
s<\min\{0,-a^{\frac13}\} \). This is the case for any non-zero \( a
\): if \( a>0 \) the solution is obtained for
\[ s=-a^{\frac13}(1+\frac34\sqrt{2})^{\frac13}, \] and if \( a<0 \)
the solution occurs when
\[ s=a^{\frac13}(-1+\frac34\sqrt{2})^{\frac13}. \]
     
Introducing \( A(t)=-\frac{(\alpha x)'}{\alpha x} \), we can express
the exterior derivatives of the defining forms via
\begin{equation}
  \begin{cases}
    \label{eq:flow13}
    d\omega=-\frac32A\gamma+\frac32\left(\alpha a(e^{135}-e^{246})+(A-\alpha)\gamma\right)\equiv-\frac32A\gamma+\beta,\\
    d\gamma=0,\\
    d\hat\gamma=A\omega^2.
  \end{cases}
\end{equation}
As \( \gamma\wedge\beta=0=\hat\gamma\wedge\beta \) and \(
\omega\wedge\beta=0 \), this implies that the constructed
one-parameter family of \( \mathrm{SU}(3) \)-structures consists of
members of type \( \mathcal W_1+\mathcal W_3 \).

The associated family of metrics takes the form
\[ g=-\frac{3\alpha
  x}{\sqrt{(3x-a)(x+a)}}\left(x\sum_{i=1}^6e^i\otimes
  e^i+\frac12(a-x)\sum_{i=1}^3(e^{2i-1}\otimes e^{2i}+e^{2i}\otimes
  e^{2i-1})\right), \] and has scalar curvature given by
\[ \mathpzc{s}=\frac{6(a^2-5x^2)}{\sqrt{(3x-a)^3(a+x)}}.\]
  
Zero scalar curvature is obtained for the solution which has \(
a=-(5+\sqrt{5}) \). Indeed, in this case the scalar curvature is zero
when \( s^3=\frac{1-\sqrt{5}}2 \).

Finally, let us remark that the associated \( \GT \)-metric is of the
form \( dt\otimes dt+g \), or, phrased more explicitly, in terms of
the parameter \( s \):
\begin{equation*}
  \begin{split}
    &\frac{12}{1+as^{-3}}ds\otimes
    ds+\frac{4s^2+as^{-1}}{\sqrt{3}}\sum_{i=1}^6e^i\otimes e^i -
    \frac{2s^2-as^{-1}}{\sqrt{3}}\sum_{i=1}^3(e^{2i-1}\otimes
    e^{2i}+e^{2i}\otimes e^{2i-1})=\\[5pt]
    &\frac{12}{1+as^{-3}}ds\otimes
    ds\\
    &\hskip20pt+\sum_{i=1}^3\left(\frac{s^2(1+as^{-3})}{\sqrt{3}}(e^{2i-1}+e^{2i})\otimes
      (e^{2i}+e^{2i-1}) + \sqrt{3}s^2(e^{2i-1}-e^{2i})\otimes
      (e^{2i}-e^{2i-1})\right).
  \end{split}
\end{equation*}
If \( a=0 \) this metric is conical whilst for \( a\neq0 \), the
metric is asymptotically conical: when \( |s|\to \infty \) it tends to
a cone metric
\[ 12ds^2+s^2\sum_{i=1}^3\left(\frac{1}{\sqrt{3}}(e^{2i-1}+e^{2i})\otimes
  (e^{2i}+e^{2i-1}) + \sqrt{3}(e^{2i-1}-e^{2i})\otimes(e^{2i}-e^{2i-1})\right) \]
over \( M \). In terms of the classification \cite{Dancer-W:painleve},
the metrics belong to the family (I).

In terms of the matrix framework, the one-parameter families of pairs
\( (Q,P) \) take the form:
\begin{equation*}
  Q=-x\diag(3,-1,-1,-1),\quad P=-\frac32\alpha x\diag(3,-1,-1,-1).
\end{equation*}
In particular, we get another way of verifying the co-coupled
condition:
\begin{equation*}
  (P^2)_0=\frac{9\alpha^2x^2}2\diag(3,-1,-1,-1),\quad R=x(a + x)^2\diag(3,-1,-1,-1). 
\end{equation*}

\section{Further examples}
\label{sec:further}

\paragraph{Metrics with \boldmath{\(\SU(2)^2\times \Delta
    \mathrm{U}(1)\ltimes \mathbb{Z}\slash2\)} symmetry.}

Following mainly \cite{Chong-al:G2contr}, we study examples that
relate our framework to certain constructions of \( \GT \)-metrics
appearing in the physics literature. Our starting point in a
one-parameter families half-flat pairs \( (\omega,\gamma) \) of the
form:
\begin{equation*}
  \begin{cases}
    \omega=p_1e^{12}+p_2e^{34}+p_3e^{56},\\
    \gamma=ae^{135}+be^{246}+q_1d(e^{12})+q_2d(e^{34})+q_3d(e^{56}).
  \end{cases}
\end{equation*}
Using the normalisation condition, we are able to express the
associated one-parameter family of metrics on \( M \) as follows:

\begin{equation}
  \begin{split} 
    \label{eq:diagmetr}
    g&=\frac{q_2q_3+aq_1}{p_2p_3}e^1\otimes e^1+\frac{q_2q_3-bq_1}{p_2p_3}e^2\otimes e^2+\frac{q_1^2-q_2^2-q_3^2-ab}{2p_2p_3}(e^1\otimes e^2{+}e^2\otimes e^1)\\
    &+\frac{q_1q_3+aq_2}{p_1p_3}e^3\otimes e^3+\frac{q_1q_3-bq_2}{p_1p_3}e^4\otimes e^4+\frac{q_2^2-q_1^2-q_3^2-ab}{2p_1p_3}(e^3\otimes e^4{+}e^4\otimes e^3)\\
    &+\frac{q_1q_2+aq_3}{p_1p_2}e^5\otimes
    e^5+\frac{q_1q_2-bq_3}{p_1p_2}e^6\otimes
    e^6+\frac{q_3^2-q_1^2-q_2^2-ab}{2p_1p_2}(e^5\otimes
    e^6{+}e^6\otimes e^5),
  \end{split}
\end{equation}
and the flow equations \eqref{eq:G2flow} read:
\begin{equation}
  \label{eq:flow-specialcase}
  \begin{cases}
    q_i'=p_i,\\
    (p_2p_3)'=\frac1{p_1p_2p_3}\left(-abq_1+(a-b)q_2q_3+q_1(q_2^2+q_3^2-q_1^2)\right),\quad\textrm{etc.}
  \end{cases}
\end{equation}

\begin{remark}
  Notice that the \( \mathbb{Z}\slash2 \) action which interchanges
  the two copies of \( S^3 \) preserves the metric \eqref{eq:diagmetr}
  provided the cohomology class \( [\gamma] \) is of the form \( a+b=0
  \), i.e., \( [\gamma]=(a,-a) \). The action interchanges metrics of
  half-flat structures with \( [\gamma]=(a,0) \) with those for which
  \( [\gamma]=(0,-a) \). The latter observation is related to the
  notion of a \emph{flop} \cite{Atiyah-al:flop}.
\end{remark}

\begin{remark}
  \label{rem:volgrowth}
  The quantity \( \sqrt{\det{g(t)}} \) can be viewed as the ratio of
  the volume of \( g(t) \) relative to a fixed background metric on \(
  \SM \). As expected, we find that
  \[ \sqrt{\det(g)}=2\sqrt{-\lmb}, \] where we have used that \(
  \tr(P^3)=-6\sqrt{-\lmb} \), by the normalisation condition
  \eqref{eq:normMatr}.
\end{remark}

A metric ansatz that has led to the discovery of new complete \( \GT
\)-metrics (see, for instance,
\cite{Brandhuber-al:G2,Cvetic-al:orientifolds}) can be expressed in
terms of the condition \( a+b=0 \). In this case, we find
\begin{equation}
  \begin{split} 
    \label{eq:triaxmetr}
    g&=\frac{q_2q_3+aq_1}{p_2p_3}(e^1\otimes e^1+e^2\otimes e^2)+\frac{q_1^2-q_2^2-q_3^2+a^2}{2p_2p_3}(e^1\otimes e^2{+}e^2\otimes e^1)\\
    &+\frac{q_1q_3+aq_2}{p_1p_3}(e^3\otimes e^3+e^4\otimes e^4)+\frac{q_2^2-q_1^2-q_3^2+a^2}{2p_1p_3}(e^3\otimes e^4{+}e^4\otimes e^3)\\
    &\frac{q_1q_2+aq_3}{p_1p_2}(e^5\otimes e^5+e^6\otimes e^6)+\frac{q_3^2-q_1^2-q_2^2+a^2}{2p_1p_2}(e^5\otimes e^6{+}e^6\otimes e^5)\\
    &=\sum_{i=1}^3a_i^2(e^{2i-1}-e^{2i})\otimes(e^{2i-1}-e^{2i})+b_i^2(e^{2i-1}+e^{2i})\otimes(e^{2i-1}+e^{2i}),
  \end{split}
\end{equation}
where
\begin{equation*}
  \begin{cases}
    \label{eq:rel-abqp}
    a_1^2+b_1^2=\frac{q_2q_3+aq_1}{p_2p_3},b_1^2-a_1^2=\frac{q_1^2-q_2^2-q_3^2+a^2}{2p_2p_3},\\
    a_2^2+b_2^2=\frac{q_1q_3+aq_2}{p_1p_3},b_2^2-a_2^2=\frac{q_2^2-q_1^2-q_3^2+a^2}{2p_1p_3},\\
    a_3^2+b_3^2=\frac{q_1q_2+aq_3}{p_1p_2},b_3^2-a_3^2=\frac{q_3^2-q_1^2-q_2^2+a^2}{2p_1p_2},
  \end{cases}
\end{equation*}
or, alternatively,

\begin{equation}
  \begin{cases}
    q_1=-a_1a_2a_3 - a_3b_1b_2 - a_2b_1b_3 + a_1b_2b_3,\\
    q_2=-a_1a_2a_3 - a_3b_1b_2 + a_2b_1b_3 - a_1b_2b_3,\\
    q_3=-a_1a_2a_3 + a_3b_1b_2 - a_2b_1b_3 - a_1b_2b_3\\
    p_2p_3=4a_2a_3b_2b_3,p_1p_3=4a_1a_3b_1b_3,p_1p_2=4a_1a_2b_1b_2,\\
    a=-b=a_1a_2a_3 - a_3b_1b_2 - a_2b_1b_3 - a_1b_2b_3.
  \end{cases}
\end{equation}
Note that, up to a sign, we have \( p_i=-2a_ib_i \).

Expressed in terms of the metric function \( a_i,b_i \), the flow
equations \eqref{eq:flow-specialcase} become:
\begin{equation*}
  \begin{cases}
    \label{eq:flow-specialcase-3axial}
    4a_1'= \frac{a_1^2}{a_3b_2} + \frac{a_1^2}{a_2b_3} - \frac{a_2}{b_3} - \frac{a_3}{b_2} - \frac{b_2}{a_3} - \frac{b_3}{a_2},\\
    4b_1'= \frac{b_1^2}{a_2a_3} - \frac{b_1^2}{b_2b_3} - \frac{a_2}{a_3} - \frac{a_3}{a_2} + \frac{b_2}{b_3} + \frac{b_3}{b_2},\\
    4a_2'= \frac{a_2^2}{a_3b_1} + \frac{a_2^2}{a_1b_3} - \frac{a_1}{b_3} - \frac{a_3}{b_1} - \frac{b_1}{a_3} -  \frac{b_3}{a1}, \\
    4b_2'= \frac{b_2^2}{a_1a_3} - \frac{b_2^2}{b_1b_3} -  \frac{a_1}{a_3} - \frac{a_3}{a_1} +\frac{b_1}{b_3} +  \frac{b_3}{b_1},\\
    4a_3'= \frac{a_3^2}{a_2b_1} + \frac{a_3^2}{a_1b_2} -  \frac{a_1}{b_2} - \frac{a_2}{b_1} -\frac{b_1}{a_2} -  \frac{b_2}{a_1}, \\
    4b_3'= \frac{b_3^2}{a_1a_2} - \frac{b_3^2}{b_1b_2} -
    \frac{a_1}{a_2} - \frac{a_2}{a_1} + \frac{b_1}{b_2} +
    \frac{b_2}{b_1}.
  \end{cases}
\end{equation*}

The complete metrics constructed by Brandhuber et al
\cite{Brandhuber-al:G2} arise as a further specialisation of this
system.  Indeed, if we take \( a_1=a_2\equiv a \) and \( b_1=b_2\equiv
b \) and set \( t=\int\frac{ds}{b_3} \), then the system
\eqref{eq:flow-specialcase-3axial} reads
\begin{equation*}
  \begin{cases}
    4\frac{\partial a}{\partial s}= \frac{a^2-a^2_3-b^2}{ba_3b_3} - \frac1{a},\\
    4\frac{\partial b}{\partial s}= \frac{b^2- a^2 -a^2_3}{aa_3b_3} + \frac1{b},\\
    2\frac{\partial a_3}{\partial s}= \frac{a_3^2 -a^2-b^2}{abb_3}, \\
    4\frac{\partial b_3}{\partial s}= \frac{b_3}{a^2} -
    \frac{b_3}{b^2},
  \end{cases}
\end{equation*}
which is the same as in \cite[Equation (3.1)]{Brandhuber-al:G2}, where
the authors find the following explicit holonomy \( \GT \)-metric:
\begin{equation}
  \label{eq:ABC}
  \begin{split}
    \frac{ds^2}{b_3^2}&+\frac{(s-\frac32)(s+\frac92)}{12}\left((e^1-e^2)\otimes(e^1-e^2)+(e^3-e^4)\otimes(e^3-e^4)\right)\\
    &+\frac{(s+\frac32)(s-\frac92)}{12}\left((e^1+e^2)\otimes(e^1+e^2)+(e^3+e^4)\otimes(e^3+e^4)\right)\\
    &+\frac{s^2}9(e^5-e^6)\otimes(e^5-e^6)+\frac{(s-\frac92)(s+\frac92)}{(s-\frac32)(s+\frac32)}(e^5+e^6)\otimes(e^5+e^6).
  \end{split}
\end{equation}
Asymptotically this is the metric of a circle bundle over a cone, in
short an \emph{ABC metric}. In terms of the classification
\cite{Dancer-W:painleve}, it belongs to the family (II).

\paragraph{Cohomogeneity one Ricci flat metrics.}
Any solution of \eqref{eq:G2flow} gives us a cohomogeneity one Ricci
flat metric on \( M\times I \). An important aspect of the
cohomogeneity one terminology is to bridge a gap between our framework
and the ``Lagrangian approach'' appearing in the physics literature
(see, e.g., \cite[Section 4]{Brandhuber-al:G2}). For example, consider
the metric \eqref{eq:triaxmetr} from the above example, assuming for
simplicity that \( a_1=a_2\equiv a \) and \( b_1=b_2\equiv b \). By
\cite{Eschenburg-W:cohom1}, we know that the shape operator \( L \) of
the principal orbit \( S^3\times S^3\subset I\times M \) satisfies the
equation \( g'=2g\circ L \). For the given metric, we find that
\[ L=\frac12\left(\begin{array}{cccccc}
    \frac{a'b+ab'}{ab} & \frac{ab'-a'b}{ab} & 0 & 0 & 0 & 0 \\
    \frac{ab'-a'b}{ab} & \frac{a'b+ab'}{ab} & 0 & 0 & 0 & 0 \\
    0 & 0 &  \frac{a'b+ab'}{ab} & \frac{ab'-a'b}{ab} & 0 & 0\\
    0 & 0 & \frac{ab'-a'b}{ab} & \frac{a'b+ab'}{ab} & 0 & 0\\
    0 & 0 & 0 & 0 & \frac{a'_3b_3+a_3b'_3}{a_3b_3} & \frac{a_3b'_3-a'_3b_3}{a_3b_3}\\
    0 & 0 & 0 & 0 & \frac{a_3b'_3-a'_3b_3}{a_3b_3} &
    \frac{a'_3b_3+a_3b'_3}{a_3b_3}
  \end{array}\right). \]
We also observe that
\begin{equation*}
  \begin{gathered}
    \tr(L)^2=\frac{(2a_3b_3ab'+2a_3b_3ba'+aba_3b'_3+abb_3a'_3)^2}{a^2b^2a_3^2b_3^2},\\
    \tr(L^2)=\frac{(2a_3^2b_3^2a^2{b'}^2+2a_3^2b_3^2b^2{a'}^2+a^2b^2a_3^2{b'}_3^2+a^2b^2b_3^2{a'}_3^2}{a^2b^2a_3^2b_3^2},\\
    \det(g)=64a^4b^4a_3^2b_3^2,\\
    \scal=-\frac18\frac{2a_3^4a^2b^2+a_3^2a^4b_3^2-8a^4b^2a_3^2+a_3^2b^4b_3^2-8b^4a^2a_3^2+2a^6b^2-4a^4b^4+2a^2b^6}{a^4b^4a_3^2}.
  \end{gathered}
\end{equation*}

In general, the Ricci flat condition can now be expressed as:
\begin{equation}
  \label{eq:cohom1-Rflat}
  L'+(\tr(L))L-\Ric=0,\quad \tr(L')+\tr(L^2)=0,
\end{equation}
combined with another equation expressing the Einstein condition for
mixed directions. If we take the trace of the first equation in
\eqref{eq:cohom1-Rflat}, and combine with the second one, we obtain
the following conservation law:
\[ (\tr(L))^2-\tr(L^2)-\scal=0. \]

As explained in \cite{Dancer-W:painleve}, the above system has a
Hamiltonian interpretation. It is this interpretation, in its Lagrangian
guise and phrased with the use of superpotentials, one frequently
encounters in the physics literature. In this setting, the kinetic and
potential energies are given by
\begin{equation*}
  T=\left((\tr(L))^2-\tr(L^2)\right)\sqrt{\det(g)},\quad V=-\scal\sqrt{\det(g)};
\end{equation*}
these definitions agree with those in \cite{Brandhuber-al:G2} up to a
multiple of \( \sqrt{\det(g)}=8a^2b^2a_3b_3 \).

In \cite{Dancer-W:superpot}, the authors provide a relevant
description of the superpotential; in classical terms this is a solution of a
time-independent Hamilton-Jacobi equation. In the concrete example,
the superpotential \( u \) can be viewed as a function of \( a_i,b_i
\). Concretely, we can take
\[
u=2\left(2a^3bb_3+2ab^3b_3-a^2a_3b_3^2+b^2a_3b_3^2+2aba_3^2b_3\right). \]
In terms of \( u \), the flow equations can then be expressed as
follows:
\[ \frac{\partial\vv{\alpha}}{\partial r}=G^{-1}\frac{\partial
  u}{\partial \vv{\alpha}}, \] where \(
\vv{\alpha}=(\ln(a),\ln(b),\ln(b_3),\ln(a_3))^T \) (assuming \(
a_i,b_i>0 \)), \( t=\int\!\!\sqrt{\det(g)}\,dr \) and
\[ G=\left(\begin{array}{cccc}2 & 4 & 2 & 2 \\
    4 & 2 & 2 & 2\\
    2 & 2 & 0 & 1 \\
    2 & 2 & 1 & 0\end{array}\right) .\] Finally, we remark that the
kinetic and potential terms can be expressed in the form
\[ \sqrt{\det(g)}T=\frac{\partial\vv{\alpha}}{\partial
  r}G\left(\frac{\partial\vv{\alpha}}{\partial r}\right)^T,\quad
\sqrt{\det(g)}V=-\frac{\partial u}{\partial
  \vv{\alpha}}G^{-1}\left(\frac{\partial u}{\partial
    \vv{\alpha}}\right)^T.\]

As a further specialisation, let us consider the case when \( a=0 \)
and \( a=a_3=\frac{t}{2\sqrt{3}} \), \( b=b_3=\frac{t}6 \); this is
the nearly-K\"ahler case. Then the shape operator is proportional to
the identity: \( L=t^{-1}I \), and the kinetic and potential terms are
\[ T=\frac{5\sqrt{3}t^4}{324},\quad V=-\frac{5\sqrt{3}t^4}{324},\]
respectively. So the total energy is zero \( T+V=0 \) for all
\(t>0\). The superpotential is the fifth oder polynomial
\[ u=\frac{13t^5}{216\sqrt{3}}.\]

\paragraph{Uniqueness: flowing along a line.}
In the case when \( (Q,P)\subset\mathcal H_0 \), the flow equations
\eqref{eq:G2flow-matr} turn out to have a unique (admissible) solution
satisfying for which \( Q \) belongs to a fixed one-dimensional
subspace.

\begin{proposition}
  \label{prop:NKunique}
  Assume \( t\mapsto(Q(t),P(t))\in \mathcal H_0 \) is a solution of
  \eqref{eq:G2flow-matr}. Then \( Q \) belongs to a fixed \( 1
  \)-dimensional subspace of \( S^2_0(\bR^4) \) if and only if the
  associated \( \GT \)-metric is the cone metric over \( \SM \)
  endowed with its nearly-K\"ahler structure.
\end{proposition}
\begin{proof}
  It is easy to see that the solution of \eqref{eq:G2flow-matr} which
  corresponds to the cone metric over \( \SM \) (with its
  nearly-K\"ahler structure) is represented by
  \begin{equation}
    \begin{cases}
      \label{eq:NKcone-matr}
      (Q(t),P(t))=(q(t)\diag(-3,1,1,1),p(t)\diag(-3,1,1,1))\in\mathcal H_0,\\
      (q(t),p(t))=-\frac{t^2}{6\sqrt{3}}(\frac{t}3,1).
    \end{cases}
  \end{equation}
  So, in this case, \( Q \) indeed belongs to a fixed \( 1
  \)-dimensional subspace of \( S^2_0(\bR^4) \).

  Conversely, let us assume we are given a solution such that
  \[ Q(t)=U(t)\diag\left(-1-a-b,a,b,1\right). \] Then the system
  \eqref{eq:G2flow-matr} reads:

  \begin{equation*}
    \begin{cases}
      \left(1 +b+ c - b^2 + c^2  + bc\right)uu'=\frac{b (-1 + c)^2 + b^2 (1 + c) - 3 c (1 + c)}{\sqrt{b c (1 + b + c)}}U,\\
      \left(1 + b +c + b^2 - c^2 + b c\right)uu'=\frac{b^2 (-3 + c) + c (1 + c) + b (-3 - 2 c + c^2)}{\sqrt{b c (1 + b + c)}}U,\\
      \left(-1 + b +c + b^2 + c^2 + b c\right)uu'=\frac{b + b^2 + c -
        2 b c - 3 b^2 c + c^2 - 3 b c^2}{\sqrt{b c (1 + b + c)}}U.
    \end{cases}
  \end{equation*}
  These equations show that there is a purely algebraic constraint to
  having a solution:
  \begin{equation*}
    \begin{cases}
      1 +b+c- b^2 + c^2 + bc=\frac{b (-1 + c)^2 + b^2 (1 + c) - 3 c (1 + c)}{\sqrt{b c (1 + b + c)}}\kappa,\\
      1 + b +c+ b^2 -c^2+ b c=\frac{b^2 (-3 + c) + c (1 + c) + b (-3 - 2 c + c^2)}{\sqrt{b c (1 + b + c)}}\kappa,\\
      -1 + b+c + b^2 + c^2 + b c =\frac{b + b^2 + c - 2 b c - 3 b^2 c
        + c^2 - 3 b c^2}{\sqrt{b c (1 + b + c)}}\kappa,
    \end{cases}
  \end{equation*}
  where \( \kappa\in \bR \). Uniqueness of the ``nearly-K\"ahler
  cone'', as a flow solution, now follows by observing that these
  algebraic equations have the following set of solutions:
  \begin{equation*}
    \begin{gathered}
      (\kappa,b,c)=(0,-1,-1), (\kappa,b,c)=(0,1,-1), (\kappa,b,c)=(0,-1,1),\\
      (\kappa,b,c)=(\frac1{\sqrt3},-\frac13,-\frac13),
      (\kappa,b,c)=(-\sqrt3,1,-3),(\kappa,b,c)=(-\sqrt3,
      -3,1),\\(\kappa,b,c)=(-\sqrt3, 1,1).
    \end{gathered}
  \end{equation*} 
  The solutions with \( \kappa=0 \) are not ``admissible'' whilst the
  remaining solutions all result in one-parameter families of pairs
  equivalent to \eqref{eq:NKcone-matr}.
\end{proof}

\section{Numerical solutions}
\label{sec:num}

As indicated in the earlier parts of this paper, previous studies of
\( \GT \)-metrics on \( M\times I \) have focused mainly on metrics
with isometry group (at least) \(\SU(2)^2\times \Delta
\mathrm{U}(1)\ltimes \mathbb{Z}\slash2 \). In addition, most of the
attention has been centred around solutions in \( \mathcal H_c \) for
\( c=(a,-a)\neq0 \).

A technique that seems effective if one is specifically looking
for complete metrics is to choose the initial values of the flow
equations \eqref{eq:G2flow-matr} to obtain a singular orbit at that
point (meaning, in our context, one whose stabilizer has positive
dimension in \(\SU(2)^2\)). This approach was adopted in
\cite{Reidegeld:Spin7,Cvetic-al:G2-Spin7} for \(\mathrm{Spin}(7)\)
holonomy. However, this final section shifts the focus of our
investigation in order to illustrate some more generic behaviour of
the flow on the space of invariant half-flat structures on \(
S^3\times S^3 \).

\paragraph{Two-function ansatz.}
We first look for solutions in \( \mathcal H_0 \) for which \( Q \)
takes the form
\[ Q(t)=\diag(-2U(t)-V(t),U(t),U(t),V(t)),\] where \( U,V \) are
smooth functions on an interval \( I \subset\bR \). A solution of
\eqref{eq:G2flow-matr} is then uniquely specified by the quadruple
\[ (U(0),V(0),U'(0),V'(0)). \]

We have solved the system for a wide range of initial conditions. A
selection of solutions are shown in Figure \ref{fig:G2sol2D}. Apart
from the nearly-K\"ahler straight line, these solutions are
new. Plotting the metric functions, we find that some of the new
metrics have one stabilising direction when \( t \to\infty \) and no
collapsing directions (they are therefore ABC metrics of the sort
mentioned in connection with \eqref{eq:ABC}). The others have
shrinking directions which cause the volume growth to slow down as
shown in Figure \ref{fig:G2sol2Dvgrowth}.

\bigbreak

\begin{figure}[ht!]
  \begin{center}
    \subfigure[Solution curves with \( (U(0),V(0)) \) fixed.]{
      \label{fig:G2sol2DqcurveA}
      \includegraphics[width=0.35\textwidth]{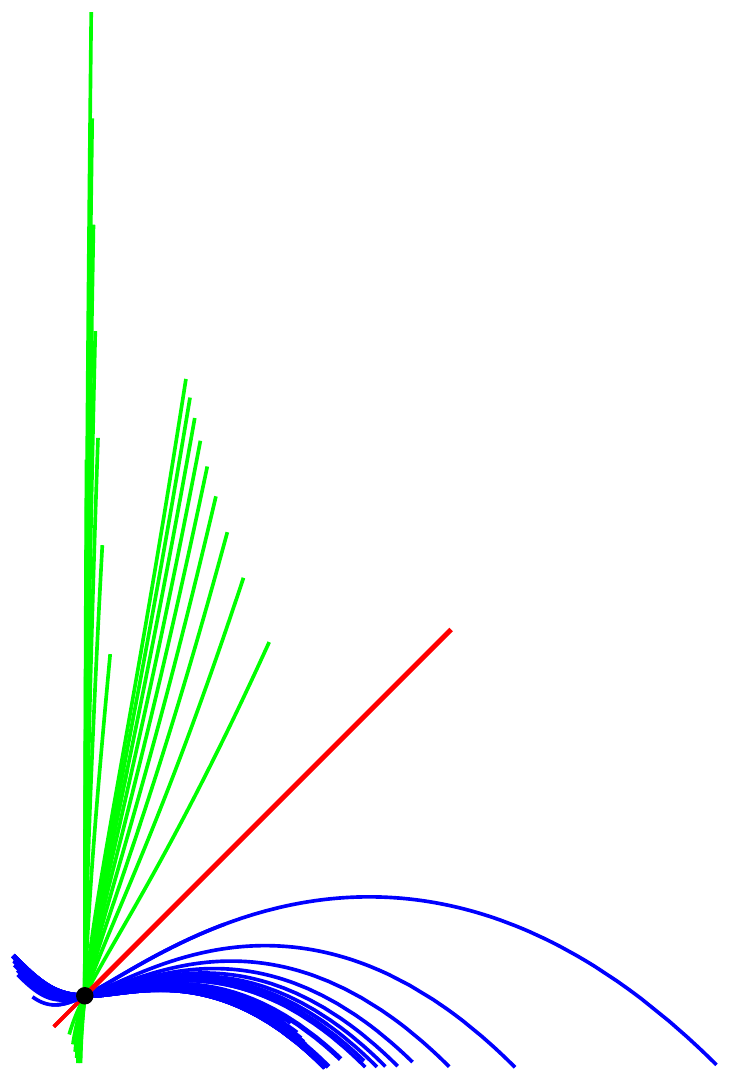}
    } \qquad \subfigure[Solution curves with \( (U'(0),V'(0)) \)
    fixed.]{
      \label{fig:G2sol2DqcurveB}
      \includegraphics[width=0.35\textwidth]{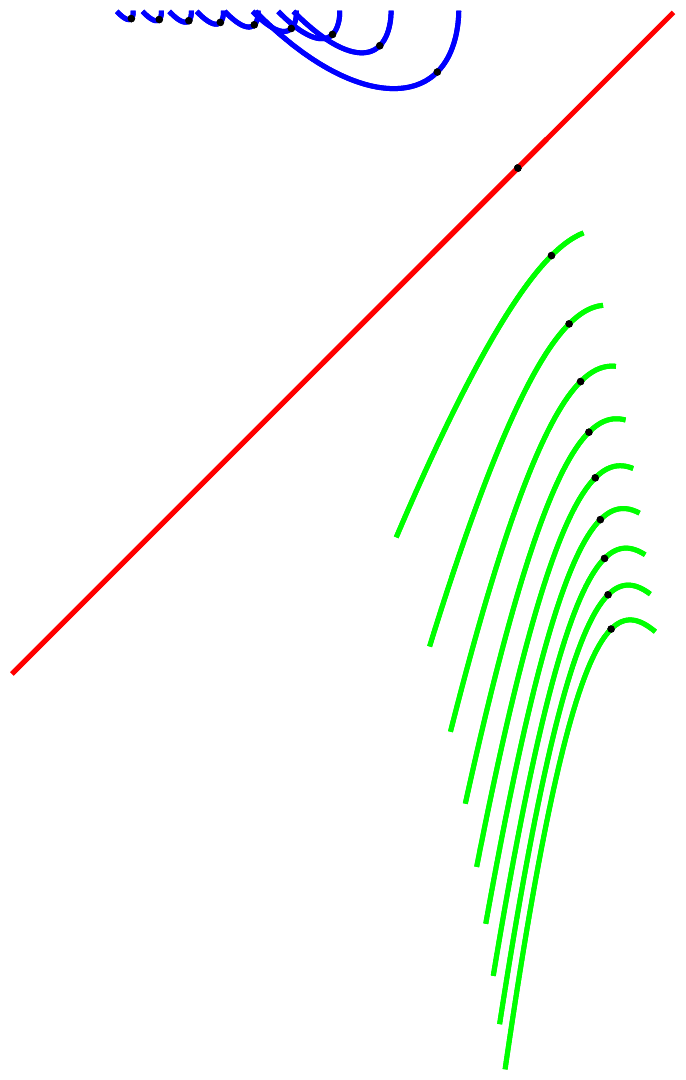}
    }\\[10pt]
    \subfigure[Volume growth for selected solutions.]{
      \label{fig:G2sol2Dvgrowth}
      \includegraphics[width=0.6\textwidth]{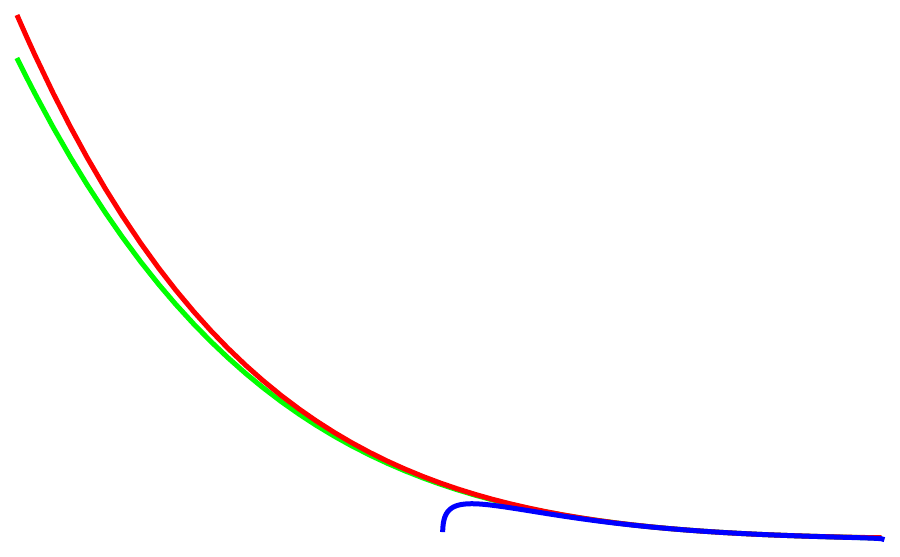}
    }
  \end{center}
  \caption{A collection of ``planar solutions'' satisfying \( a=0=b
    \). The solution curves are given in terms of \(
    t\mapsto(U(t),V(t)) \) whilst the volume growth refers to \(
    t\mapsto\sqrt{-\lambda(t)} \).}
  \label{fig:G2sol2D}
\end{figure}

More precisely, in the case \( U(0)=V(0) \), the normalisation forces
\( Q'(0) \), written as \( (x,y)=(U'(0),V'(0)) \), to lie on the curve
\begin{equation}
  \label{eq:nromcurve}
  x(x+y)^2= -2\sqrt{3},
\end{equation}
which has two branches separated by the line \( x+y=0 \). One branch
corresponds to positive-definite metrics, including the
nearly-K\"ahler solution
\begin{equation}
  \label{eq:nkc}
  x=y=\nu,\quad\hbox{where}\quad\nu=-3^{1/6}/2^{1/3}=-0.953\ldots
\end{equation}
The ABC metrics are those for which \(\nu<x<0\), and appear on the top
left of the nearly-K\"ahler line in Figure \ref{fig:G2sol2DqcurveA},
in green in the coloured version.

When \( U(0)\neq V(0) \), the nearly-K\"ahler solution is
excluded. Nevertheless, the overall picture remains valid, meaning one
branch of the normalisation curve corresponds to positive-definite
metrics, and this branch itself has two half pieces, one corresponding
to ABC curves and one to the other solutions.

In the trace-free case, \( a=0=b \), all solutions degenerate at a
point \( t_0 \). The ABC solutions are ``half complete'', meaning that
away from the degeneration they are complete in one direction of
time. (See \cite{Apostolov-S:G2,Chiossi-F:G2} for other examples of
half-complete \( \GT \)-metrics). The other solutions reach another
degeneracy point \( t_1 \) in finite time. The singularity at \( t_0
\) cannot be resolved. In particular, it is not possible to find
complete \( \GT \)-metrics. One way to circumvent this issue is to
consider flow solutions for which \( [\gamma]\neq0 \); solutions of
this form include the metrics discovered by Brandhuber et al
\cite{Brandhuber-al:G2}.

\paragraph{Three-function ansatz.}
Now, turning to ``less symmetric'' \( \GT \)-metrics, we consider for
solutions in \( \mathcal H_0 \) with \( Q \) of the (generic) form:
\[ Q(t)=\diag(-U(t)-V(t)-W(t),U(t),V(t),W(t)),\] where \( U,V,W \) are
smooth functions on an interval \( I \subset\bR \). A solution of
\eqref{eq:G2flow-matr} is then uniquely specified by the sextuple
\[ (U(0),V(0),W(0),U'(0),V'(0),W'(0)). \]

As in the case of planar solutions, we have solved the flow equations
for a large number of initial conditions. In contrast with the planar
case, we have not been able to find metrics with one stabilising
directions as \( t \to\pm\infty \).

We shall confine our presentation to the class of solutions with the
same initial point
\[ (U(0),V(0),W(0))=(1,1,1) \] as the nearly-K\"ahler solution, but
with varying velocity vector
\begin{equation}
  \label{eq:initialvel3d}
  (x,y,z)=(U'(0),V'(0),W'(0)).
\end{equation}
Similar to the planar case, the flow lines are governed by the
normalization condition, and \eqref{eq:nromcurve} is replaced by the
cubic surface
\begin{equation}
  \label{eq:spaceconstraint}
  (x+y)(x+z)(y+z)=-4\sqrt3.
\end{equation}
The asymptotic planes corresponding to the vanishing of \(
x+y,\,x+z,\,y+z \) separate the surface into four hyperboloid-shaped
components, and only the one with all factors negative is relevant to
our study of positive-definite metrics with holonomy \( \GT \). The
nearly-K\"ahler solution \(x=y=z=\nu\) (cf.~\eqref{eq:nkc})
corresponds to its centre point.

\bigbreak

\begin{figure}[ht!]
  \begin{center}
    \subfigure[Side view with diagonal nearly-K\"ahler line.]{
      \label{fig:G2sol3DqcurveA}
      \includegraphics[width=0.8\textwidth]{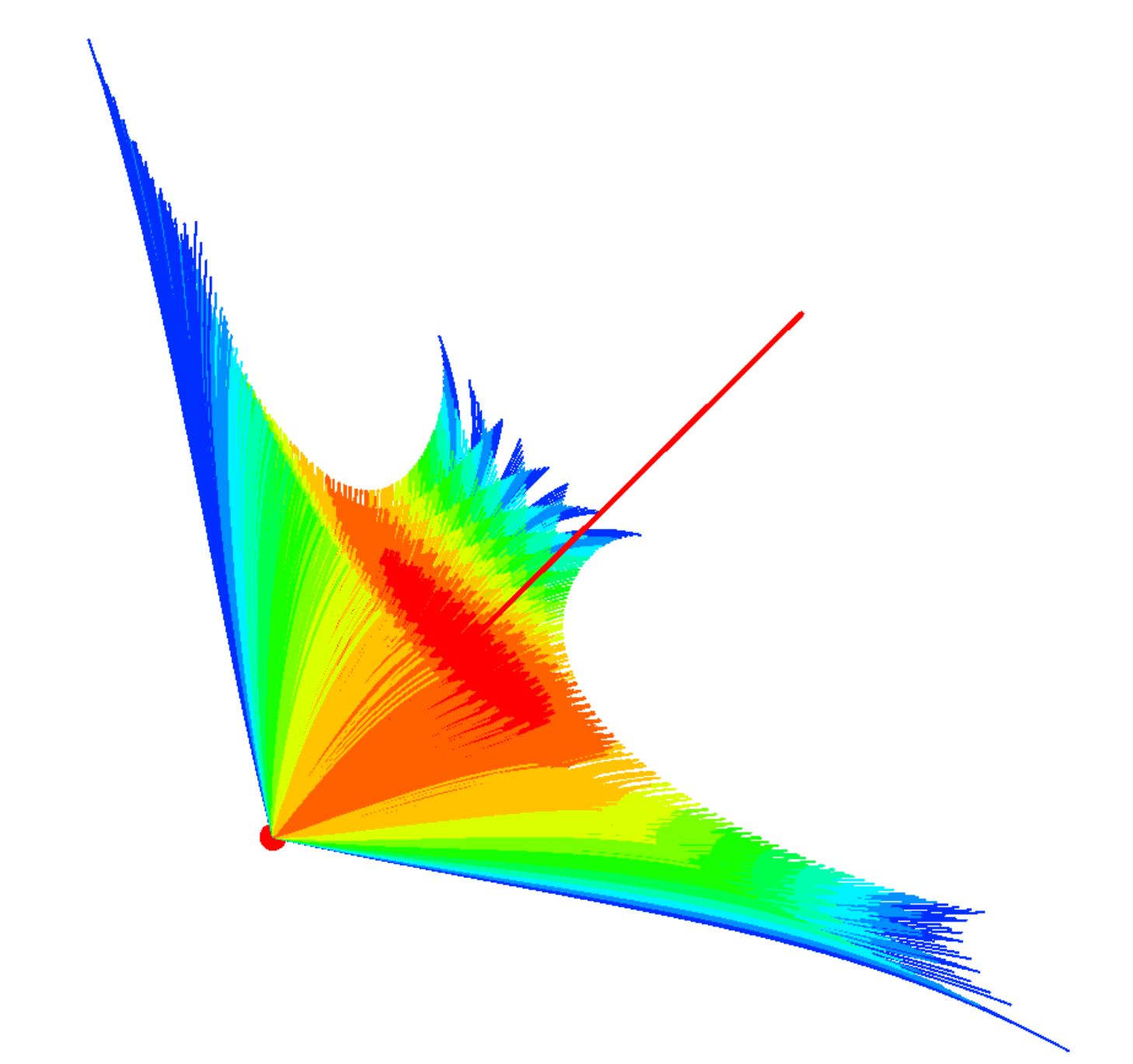}
    }\\[10pt]
    \subfigure[Looking down the line.]{
      \label{fig:G2sol3DqcurveB}
      \includegraphics[width=0.54\textwidth]{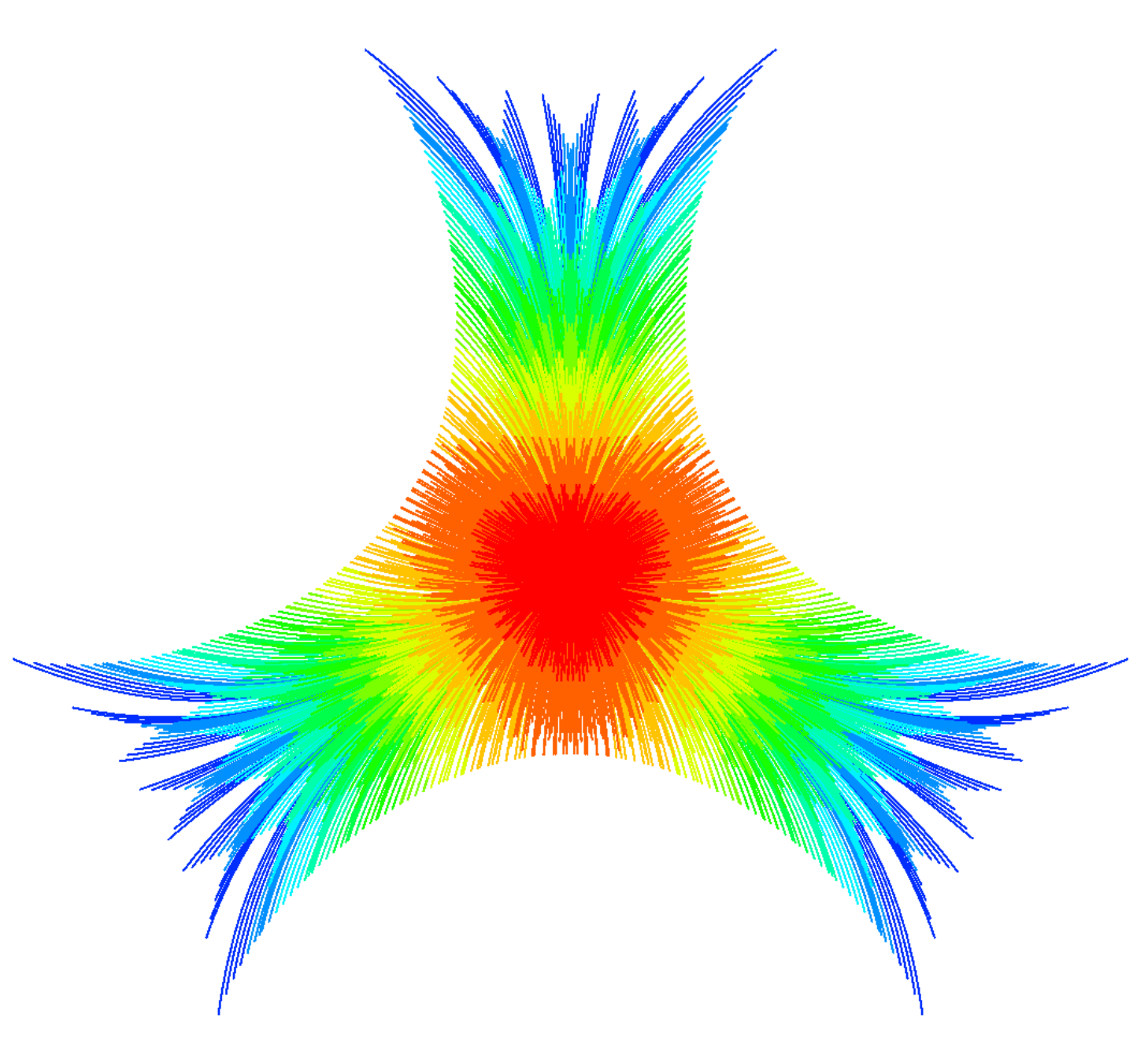}
    }\kern-5pt \subfigure[Planar ABC solutions.]{
      \label{fig:G2sol3DqcurveC}
      \includegraphics[width=0.45\textwidth]{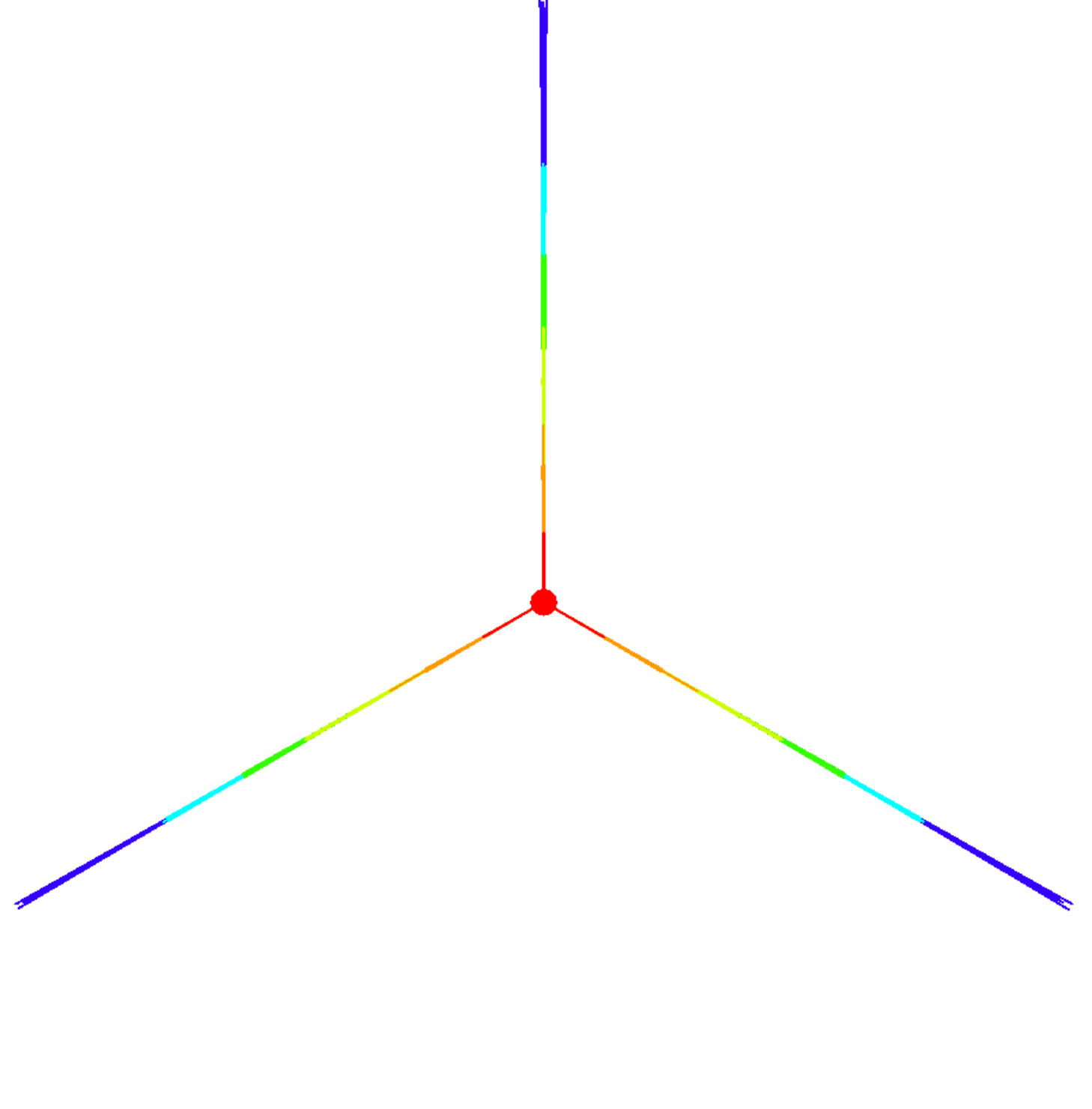}
    }
  \end{center}
  \caption{Families of space curve solutions satisfying \( a=0=b
    \). The solution curves are given in terms of \(
    t\mapsto(U(t),V(t),W(t)) \).}
  \label{fig:G2sol3D}
\end{figure}

Families of solutions are shown in Figure~\ref{fig:G2sol3D} which,
like those in Figure~\ref{fig:G2sol2D}, were plotted using
\textsl{Mathematica} and the command \textsl{NDSolve}. To obtain the
curves, it was convenient to further reduce attention to the case in
which \( x,y,z \) are all negative. The corresponding subset of
\eqref{eq:spaceconstraint} is now a curved triangle \( \mathscr{T} \)
with truncated vertices.  By issuing a plotting command for \(
\mathscr T \), we obtained an abundant sample of mesh points to feed
into \eqref{eq:initialvel3d} as initial values. One can then regard
each curve as the continuing trajectory of a particle launched towards
a point of \( \mathscr T \), which fits in close to the apex of
Figure~\ref{fig:G2sol3DqcurveA}.

\smallbreak

All the solutions, apart from the central nearly-K\"ahler one, are
new. They tend to have shrinking directions, causing the volume growth
to slow down. The \( 5250 \) solution curves in
Figure~\ref{fig:G2sol3DqcurveA} are plotted for the range \(
-0.97\leqslant t \leqslant 0 \) since many develop singularities close
to \( t=-1 \) (and close to \( t=0.2 \) though positive \( t \) is not
shown). In the coloured ``cocktail umbrella'' picture, they are
separated into groups distinguished by the value of the function \(
x^2+y^2+z^2 \) of the initial condition, with the nearly-K\"ahler line
\(x=y=z\) and its close neighbours in red.  Solutions resulting from
one of the coordinates being positive can be short-lived in comparison
to the others, leading to less coherent plots, and this is why they
are absent.

The view looking down the nearly-K\"ahler line from a point \(
(u,u,u)\) with \( u\gg1 \) is shown in
Figure~\ref{fig:G2sol3DqcurveB}. The \( \mathbb{Z}\slash{3\mathbb{Z}}
\) symmetry obtained by permuting the coordinates is evident. The
splitting behaviour at the three ``ends'' is to some extent
artificial, reflecting as it does the truncation that has resulted
from our decision to restrict attention to the negative octant.

The ABC two-function solutions of Figure~\ref{fig:G2sol2DqcurveA} in the
previous subsection arise when two of \( x,y,z \) coincide and assume
a common value greater than \(\nu\). The projection of these planar
curves orthogonal to the nearly-K\"ahler line can be seen in
Figure~\ref{fig:G2sol3DqcurveC}. Computations confirm that, unlike the
generic curves of Figure~\ref{fig:G2sol3DqcurveB} emanating from
\((1,1,1)\), these can be extended for all \(t\to-\infty\).

\smallbreak

In addition to the solutions in \( \mathcal H_0=\mathcal H_{(0,0)} \),
we have investigated solutions in \( \mathcal H_{(1,-1)} \). Regarding
the asymptotic behaviour of the associated \( \GT \)-metrics, the
overall picture appears not dissimilar to the one we have described by
deforming the nearly-K\"ahler velocity. Taking account also of the
numerical analysis in \cite{Cvetic-al:G2-Spin7}, we conjecture that
the only solutions that can be extended for \( t\to-\infty \) or
\(t\to\infty\) lie in a plane.

\paragraph*{Acknowledgements.}
Both authors thank Mark Haskins for discussions that helped initiate
this research, and in particular for bringing \cite{Hengesbach:phd} to
their attention. The first author gratefully acknowledge financial
support from the \textsc{Danish Council for Independent Research,
  Natural Sciences}.

\bigskip

\begin{small}
  \parindent0pt\parskip\baselineskip

  Thomas Bruun Madsen and Simon Salamon

  Department of Mathematics, King's College London,\\
  Strand, London WC2R 2LS, United Kingdom.
  
  \textit{E-mail}: thomas.madsen@kcl.ac.uk, simon.salamon@kcl.ac.uk
\end{small}

\end{document}